\documentclass[12pt,oneside]{amsart}

\usepackage{graphicx}
\usepackage[utf8]{inputenc}
\usepackage{hyperref}
\newtheorem{thm}{Theorem}[section]
\newtheorem{cor}[thm]{Corollary}

\theoremstyle{definition}

\theoremstyle{remark}
\newtheorem{rem}[thm]{Remark}
\numberwithin{equation}{section}

\begin{document}

\title[Feedback control for the Ginzburg-Landau equation]{Finite-parameter feedback control for stabilizing the complex Ginzburg-Landau equation}%
\author{Jamila Kalantarova}%
\address{Department of Mathematics, Izmir Institute of Technology, İzmir, Turkey}%
\email{jamilakalantarova@iyte.edu.tr}%

\author{Türker Özsarı}%
\address{Department of Mathematics, Izmir Institute of Technology, İzmir, Turkey}%
\email{turkerozsari@iyte.edu.tr}%

\thanks{Both researchers in this work were supported by TÜBİTAK 3501 Career Grant 115F055}%
\begin{abstract}
  In this paper, we prove the exponential stabilization of solutions for complex Ginzburg-Landau equations using finite-parameter feedback control algorithms, which employ finitely many volume elements, Fourier modes or nodal observables (controllers).  We also propose a feedback control for steering solutions of the Ginzburg-Landau equation to a desired solution of the non-controlled system.  In this latter problem, the feedback controller also involves the measurement of the solution to the non-controlled system.
\end{abstract}
\maketitle
\tableofcontents
\section{Introduction}  In this paper, we study the complex Ginzburg-Landau equation (CGLE), which is a mathematical model to describe near-critical instability waves such as a reaction-diffusion system near a Hopf-bifurcation.  Specific applications of this equation include nonlinear waves, second-order phase transitions, superconductivity, superfluidity, Bose-Einstein condensation and liquid crystals. See \cite{AK2002} and the references therein for an overview of several phenomena described by the CGLE.

The general form of the CGLE is written as
\begin{equation}\label{CGLE}
u_t-(\lambda+i\alpha)\Delta u+(\kappa+i\beta)|u|^p u-\gamma u=0,\quad x\in\Omega,t>0,
\end{equation} where $u$ denotes the complex oscillation amplitude, and $\beta\in\mathbb{R}$ and $\alpha\in\mathbb{R}$ are the (nonlinear) frequency and (linear) dispersion parameters, respectively.  The constants $\lambda$ and $\kappa$ are assumed to be strictly positive.  $\Omega$ is a general domain in $\mathbb{R}^n$ and $p>0$ is the source-power index.  \eqref{CGLE} can be associated with Dirichlet, Neumann, periodic, or mixed boundary conditions depending on the physical situation.  Note that equation \eqref{CGLE} simultaneously generalizes the real reaction diffusion equation and the nonlinear Schrödinger equation,  which can be obtained in the limit as the parameter pairs $(\alpha,\beta)$ and $(\lambda,\kappa)$ tend to zero, respectively.

It is well-known that if the \emph{Benjamin-Feir-Newell} stability criteria ($\alpha\beta > -1$) fails to hold, then CGLE might possess unstable solutions such as the trivial solution and chaos might be observed.  The Stoke's solution $\displaystyle u(x,t)\equiv \sqrt{\frac{\gamma}{\kappa}}e^{-i\frac{\beta\gamma}{\kappa}t}$ is another example of a space independent periodic solution (for $p=2$) whose perturbations might be unstable \cite{StuPri78}.  Motivated by these observations, we want to study the stabilization problem for the CGLE.  We will be interested only in the case $\gamma>0$, since otherwise solutions already decay to zero, and there is no room for instability.

Controlling chaotic behavior has been one of the major subjects in the theory of evolution equations, and many approaches have been developed.  One such approach involves using local or global interior control terms.  Others involve external (boundary) controls, especially in models where it is difficult or impossible to access the medium.  Using feedback type controls is a common tactic to suppress the chaotic behavior and bring solutions back to a stable state. However, non-feedback type controls (open loop control systems) are also used for steering solutions to or near a desired state.  Exact, null, or approximate controllability models are some examples.

Regarding the stabilization of the unstable solutions of the Ginzburg-Landau equation, using an internal feedback has been a common technique.  From this point of view, both global (space independent) and local (space dependent) controls were used.  At the beginning, time-delay local feedback control mechanisms were used (see \cite{BM2004}, \cite{BS1996}, \cite{BB2006}, \cite{GN2006}). Then, a linear combination of spatially translated and time-delay local feedback control terms were introduced. For example \cite{MoSi2004} and \cite{Posi2007} used this technique to stabilize the one and two dimensional Ginzburg-Landau equations with cubic nonlinearity, respectively.  There are also some works which combine both local and global type feedback controls where a local control is by itself not sufficient to control the turbulence (see, for example \cite{StiBet2010} and \cite{StiAlf2007}).  However, in some studies (e.g., \cite{BM1996}, \cite{RNG2007}), only global feedback controls were shown to be effective, too.

Contrary to the internal feedback control mechanisms mentioned above, boundary controls which are obtained by using the so called "back-stepping" methodology were also used to stabilize the solutions of Ginzburg-Landau evolutions, see for instance \cite{AamSmyKrs2005} and \cite{AamSmyKrs2007} for stabilization of the linearized one-dimensional Ginzburg-Landau equation from the boundary.

Using non-feedback type controls (i.e., open loop control systems) is another method to steer solutions of a system to a desired state preferably in small time.  One such choice for the desired state is the zero state in which case one talks about the null-controllability.  For example, \cite{RosZha2009} proved the null-controllability of the solutions of the Ginzburg-Landau equation both from the interior and the boundary via a non-feedback type control.

It is well-known that the Ginzburg-Landau equation has a finite dimensional asymptotic in-time behavior \cite{GH87}.  In other words, there is a finite number of degrees of freedom for the Ginzburg-Landau model.  There has been some recent work utilizing this type of finite dimensionality for other dissipative systems to construct feedback controls that only use finitely many volume elements, finitely many Fourier modes, or finitely many nodal observables. For example, \cite{AT2014} studied the one dimensional cubic reaction-diffusion equation, also known as the Chafee-Infante equation. The authors presented a unified approach that can be applied to a large class of nonlinear partial differential equations including the CGLE that we study here.  The study carried out in \cite{AT2014} was important since it pointed out to the fact that the finite-dimensional asymptotic behaviour is sufficient for constructing feedback controls for most dissipative dynamical systems. See also \cite{KT2016} for a similar discussion of the nonlinear wave equation.  Motivated by these recent works, we study the complex Ginzburg-Landau equation subject to a feedback control which uses only finitely many determining systems of the parameters mentioned above.

More precisely, we study the following feedback control problems in this work:

\begin{enumerate}
  \item $L^2-$stabilization of the one dimensional CGLE model with finitely many volume elements.
  \item Both $L^2$ and $H^1$-stabilization of the CGLE model with finitely many Fourier modes.
  \item Steering solutions of the CGLE model: (i) to any solution (ii) to exponentially decaying solutions.
  \item $L^2-$stabilization of the one dimensional CGLE model with finitely many nodal observables.
\end{enumerate}
\begin{rem}[\emph{A few words on the global well-posedness}]
Our proofs are based on the multiplier technique and intrinsic properties of the feedback control.  The multiplier method in our proofs can be easily justified by classical methods where one works on approximate solutions first and then passes to the limit in the energy estimates.  The approximate solutions as well as the global solvability of the original models we study here can be obtained by using different techniques.  One method is to use the maximal monotone operator theory where various terms in the equation are first replaced by their Yosida approximations; see \cite{OY2010} and the references therein.  Another approach is of course using the Galerkin procedure where the infinite dimensional model is projected on a finite dimensional subspace.  Most recently, some $L^p-L^q$ estimates have been proved on the corresponding evolution operator of the Ginzburg-Landau equation \cite{SYY2016}, from which one can also obtain the solvability of solutions.  We will omit the details of these procedures in this paper, since the additional feedback control terms that we use here do not add any extra difficulties to the well-posedness problem.  Hence, in all of our results we will simply assume the existence of a sufficiently nice solution (in time and space).  Depending on the model posed, solutions will be assumed to be at $L^2$, $H^1$, or $H^2$ levels in space.
\end{rem}
\section{$L^2-$stabilization with finite volume elements}  In this section, we consider the Ginzburg-Landau equation with finite volume elements feedback control on a bounded interval $(0,L)$ with homogeneous Neumann boundary conditions at both ends of the domain:
\begin{equation}\label{a1}
u_t-(\lambda+i\alpha)u_{xx}+(\kappa+i\beta)|u|^p u-\gamma u=-\mu\sum_{k=1}^{N}\overline{u}_k\chi_{J_k}(x), x\in (0,L), t>0,
\end{equation}
\begin{equation}\label{a2}
u_x(0,t)=u_x(L,t)=0, \quad t>0,
\end{equation}
\begin{equation}\label{initial}
u(x,0)=u_0(x), x\in (0,l),
\end{equation} where $\lambda,\kappa,\gamma>0,$ $\alpha,\beta\in \mathbb{R}$, $\displaystyle J_k\equiv \left[\frac{(k-1)L}{N},\frac{kL}{N}\right)$, $\displaystyle\overline{u}_k\equiv \frac{1}{|J_k|}\int_{J_k}udx$, and $\chi_{J_k}$ is the characteristic function on $J_k$ for $k=1,2,...,N$.  The right-hand side, which involves the local averages (observables) $\overline{u}_k$, is regarded as a feedback controller.\\

In what follows, we will use the following equivalent definition of $H^1(0,L)$-norm for convenience.
$$\|u\|_{H^1(0,L)}^2\equiv \frac{1}{L^2}\|u\|_{L^2(0,L)}^2+\|u_x\|_{L^2(0,L)}^2.$$
\begin{thm}Let $u$ be a sufficiently smooth solution of \eqref{a1}-\eqref{initial} with \begin{equation}\label{Assm1}\frac{1}{N^2}<\min\left\{1-\frac{4\gamma}{\mu},\frac{4\lambda}{\mu L^2}\right\}.\end{equation} Then $$\|u(t)\|_{L^2(0,L)}^2\le e^{-\mu\left(\frac{1}{2}-\frac{2\gamma}{\mu}-\frac{1}{2N^2}\right) t}\|u_0\|_{{L^2(0,L)}}^2$$ for $t\ge 0$.\end{thm}
\begin{proof}Taking the $L^2$-inner product of $\eqref{a1}$ with $u$ we get

\begin{multline}\label{a3}
\int_{0} ^{L} u_{t}\bar udx+ (\lambda +i\alpha)\int_{0}^{L}|u_{x} |^2 dx+(\kappa+i\beta)\int_{0}^{L}|u|^{p+2} dx-\gamma\int_{0}^{L}|u|^2dx\\=-\mu\int_{0}^{L}I_h (u)\bar udx,
\end{multline} where $I_h(u)\equiv \sum_{k=1}^{N}\overline{u}_k\chi_{J_k}(x)$.  The feedback operator $I_h$ is indeed an interpolant operator approximating the inclusion $H^1(0,T)\hookrightarrow L^2(0,L)$.  More precisely, the following Bramble-Hilbert type inequality (see \cite[Proposition 2.1]{AT2014}) holds true. \begin{equation}\label{interpolantineq}\|u-I_h(u)\|_{L^2(0,L)}\le h\|u\|_{H^1(0,L)}\end{equation} where $h=\frac{L}{N}$ is the step size.
Writing
$$I_h(u)\overline{u}=(I_h(u)-u)\bar u+|u|^2,$$ taking two times the real part of  \eqref{a3}, and using the Cauchy-Schwartz inequality, we obtain
\begin{multline}
\frac{d}{dt}\int_{0} ^{L} |u|^2 dx+ 2\lambda \int_{0}^{L}|u_{x}| ^2 dx+2\kappa\int_{0}^{L}|u|^{p+2} dx-2\gamma\int_{0}^{L}|u|^2dx\\
\leq-\mu\int_{0}^{L} |u|^2 dx+\mu\left(\int_{0}^{L}|u|^2 dx\right)^{\frac{1}{2}}\left(\int_{0}^{L}|u-I_h(u)|^2 dx\right)^{\frac{1}{2}}.
\end{multline}
Applying Young's inequality,
\begin{multline}
\frac{d}{dt}\int_{0} ^{L} |u|^2 dx+ 2\lambda \int_{0}^{L}|u_{x}| ^2 dx+2\kappa\int_{0}^{L}|u|^{p+2} dx-2\gamma\int_{0}^{L}|u|^2dx\\
\leq-\mu\int_{0}^{L} |u|^2 dx+\frac{\mu}{2}\int_{0}^{L}|u|^2 dx+\frac{\mu}{2}\|u-I_h(u)\|_{L^2(0,L)}^2 .
\end{multline}
Using the definition of the $H^1(0,L)$-norm and the inequality \eqref{interpolantineq} we obtain,
\begin{multline}\label{star}
\frac{d}{dt}\int_{0} ^{L} |u|^2 dx+ 2\lambda \int_{0}^{L}|u_{x}| ^2 dx+2\kappa\int_{0}^{L}|u|^{p+2} dx-2\gamma\int_{0}^{L}|u|^2dx\\
\leq-\frac{\mu}{2}\int_{0}^{L}| u|^2 dx+\mu\frac{h^2}{2}\left(\frac{1}{L^2}\int_{0}^{L}|u|^2 dx+\int_{0}^{L}|u_x| ^2 dx\right).
\end{multline}  Summing up the terms in the above inequality, we get
\begin{multline}\label{star}
\frac{d}{dt}\int_{0} ^{L} |u|^2 dx+ \left(2\lambda-\frac{\mu h^2}{2}\right) \int_{0}^{L}|u_{x}| ^2 dx+\left(-2\gamma+\frac{\mu}{2}-\frac{\mu h^2}{2L^2}\right)\int_{0}^{L}|u|^2dx\\
\le -2\kappa\int_{0}^{L}|u|^{p+2} dx.
\end{multline}  Since $\kappa>0$, we have
\begin{equation}
\frac{d}{dt}\int_{0} ^{L} |u|^2 dx+\mu(-\frac{2\gamma}{\mu}+\frac{1}{2}-\frac{h^2}{2 L^2})\int_{0}^{L}|u|^2 dx+\mu(\frac{2\lambda}{\mu}-\frac{ h^2}{2})\int_{0}^{L}|u_x| ^2 dx\leq 0.
\end{equation}
Setting $\displaystyle\nu:=-\frac{2\gamma}{\mu}+\frac{1}{2}-\frac{h^2}{2 L^2},$ and $\displaystyle m:=\frac{2\lambda}{\mu}-\frac{h^2}{2}$, we can write
\begin{equation}
\frac{d}{dt}\|u(t)\|_{L^2(0,L)} ^2+\mu\nu\left(\|u(t)\|_{L^2(0,L)} ^2 +\frac{m}{\nu}\|u_x(t)\|_{L^2(0,L)} ^2\right)\leq 0.
\end{equation} By assumption \eqref{Assm1}, we can drop the last term at the left-hand side of the above inequality and deduce the rapid decay of solutions in the $L^2-$sense:
$$\|u(t)\|_{L^2(0,L)} ^2\leq e^{-{\mu \nu t}}\|u(0)\|_{L^2(0,L)} ^2$$ for $t\ge 0$.
\end{proof}

\begin{rem}Note that the Neumann boundary condition does not play a major role here.  The same result also holds for Dirichlet or mixed boundary conditions.\end{rem}
\section{Stabilization with finitely many Fourier modes}
In this section, we consider the Ginzburg-Landau equation with finitely many Fourier modes feedback control on a bounded domain $\Omega\subset \mathbb{R}^n$ with homogeneous Dirichlet boundary conditions at both ends of the domain:

\begin{equation}\label{a1fm}
u_t-(\lambda+i\alpha)\Delta u+(\kappa+i\beta)|u|^p u-\gamma u=-\mu\sum_{k=1}^{N}(u,\omega_k)\omega_k,\quad x\in \Omega,t>0,
\end{equation}
\begin{equation}\label{a2fm}
u|_{\partial \Omega}=0, \quad x\in\partial\Omega, t>0,
\end{equation}
\begin{equation}\label{a3fm}
u(x,0)=u_0(x),\quad x\in \Omega,
\end{equation} where $\lambda,\kappa,\gamma>0,$ $\alpha,\beta\in \mathbb{R}$ and $\omega_k$'s denote the orthonormal set of eigenfunction of $-\Delta$ in $L^2(\Omega)$ with the respective eigenvalues $\lambda_k$.  It is well known that $\lambda_k$'s satisfy $0<\lambda_1\le \lambda_2\le ...\le \lambda_k\le\lambda_{k+1}\le...$ and $\lambda_k\rightarrow \infty$ as $k\rightarrow \infty$.\\

We prove the stabilization at both $L^2$ and $H^1$ levels.

\begin{thm}[$L^2$-decay]Let $u$ be a sufficiently smooth solution of \eqref{a1fm}-\eqref{a3fm} with $\mu\ge \gamma$, and $N$ be big enough that $\displaystyle\lambda_{N+1}>\frac{\gamma}{\lambda}$.  Then, there exists $\omega>0$ such that $$\|u(t)\|_{L^2(\Omega)}^2\leq e^{-\omega t}\|u_0\|_{L^2(\Omega)}^2$$ for $t\ge 0$. Indeed, we have $\omega=2(\lambda-\gamma\lambda_{N+1}^{-1})\lambda_1.$\end{thm}

\begin{thm}[$H^1$-decay]Let $u$ be a sufficiently smooth solution of \eqref{a1fm}-\eqref{a3fm} with $\mu\ge \gamma$ and $N$ be big enough that $\displaystyle\lambda_{N+1}>\frac{\gamma}{\lambda}$.  Then, for all $0<\delta<\omega=2(\lambda-\gamma\lambda_{N+1}^{-1})\lambda_1$, there exists some $C>0$ such that $$\|\nabla u(t)\|_{L^2(\Omega)}^2\leq Ce^{-\delta t}\|\nabla u_0\|_{L^2(\Omega)}^2$$ for $t\ge 0$ provided that $p<4/n$.  On the other hand, if $p=4/n$, the same result holds true if $\|u_0\|_{L^2(\Omega)}$ is sufficiently small.\end{thm}

\subsection{Proof of $L^2$-stabilization}

Taking the $L^2$-inner product of $\eqref{a1fm}$ with $u$ we get

\begin{multline*}
\begin{split}
(u_t,u)+(\lambda & +i\alpha)\int_{0}^{L}|{\nabla u}|^2dx  +(\kappa+i\beta)\int_{0}^{L}|u|^{p+2} dx-\gamma\int_{0}^{L}|u|^2dx
\\ &=-\mu\sum_{k=1}^{N}(u,\omega_k)\int_{0}^{L}\omega_k\bar{u}dx
=-\mu\sum_{k=1}^{N}(u,\omega_k)(\omega_k,u)
=-\mu\sum_{k=1}^{N}|(u,\omega_k)|^2.
\end{split}
\end{multline*}
Taking two times the real part of  the above and using Parseval's identity, we obtain
\begin{multline}
\frac{d}{dt}\int_{\Omega} |u|^2 dx+ 2\lambda \int_{\Omega}|{\nabla u}| ^2 dx+2\kappa\int_{\Omega}|u|^{p+2} dx-2\gamma\sum_{k=1}^{N}|(u,\omega_k)|^2-2\gamma\sum_{k=N+1}^{\infty}|(u,\omega_k)|^2\\
\leq-2\mu\sum_{k=1}^{N}|(u,\omega_k)|^2.
\end{multline} Therefore, assuming $\mu\geq\gamma$, we have
\begin{multline}
\frac{d}{dt}\int_{\Omega} |u|^2 dx+ 2\lambda \int_{\Omega}|{\nabla u}| ^2 dx+2\kappa\int_{\Omega}|u|^{p+2} dx-2\gamma\sum_{k=N+1}^{\infty}|(u,\omega_k)|^2\\
\leq-2(\mu-\gamma)\sum_{k=1}^{N}|(u,\omega_k)|^2\leq 0.
\end{multline}  Since, $\kappa>0$, the above inequality reduces to
\begin{equation}\label{P1}
\frac{d}{dt}\|u(t)\|_{L^2(\Omega)}^2+2\lambda\|\nabla u\|_{L^2(\Omega)}^2-2\gamma\sum_{k=N+1}^{\infty}|(u,\omega_k)|^2\leq 0.
\end{equation}
Using the following Poincaré type inequality,
$$\sum_{k=N+1}^{\infty}|(u,\omega_k)|^2\leq\lambda_{N+1} ^{-1}\|\nabla u\|_{L^2(\Omega)}^2,$$
we get from \eqref{P1}:

\begin{equation}
\frac{d}{dt}\|u(t)\|_{L^2(\Omega)}^2+2(\lambda-\gamma\lambda_{N+1} ^{-1})\|\nabla u(t)\|_{L^2(\Omega)}^2\leq 0.
\end{equation}
We use the Poincaré inequality $\|\nabla u\|_{L^2(\Omega)}^2\geq \lambda_1\|u\|_{L^2(\Omega)}^2$ to get
\begin{equation}
\frac{d}{dt}\|u(t)\|_{L^2(\Omega)}^2+2(\lambda-\gamma\lambda_{N+1} ^{-1})\lambda_1\|u(t)\|_{L^2(\Omega)}^2\leq 0.
\end{equation}
Thus we have
\begin{equation}
\|u(t)\|_{L^2(\Omega)}^2\leq e^{-2(\lambda-\gamma\lambda_{N+1}^{-1})\lambda_1 t}\|u_0\|_{L^2(\Omega)}^2.
\end{equation}\\

\subsection{Proof of $H^1$-stabilization}  We will consider the case $p< \frac{4}{n}$.
Taking the $L^2$-inner product of \eqref{a1fm} with $-\Delta u$ we get
\begin{multline}
-\int_{\Omega}u_t\Delta \bar{u}dx+(\lambda +i\alpha)\int_{\Omega}|{\Delta u}|^2dx-(\kappa+i\beta)\int_{\Omega}|u|^{p}u\Delta\bar{u}dx+\gamma\int_{\Omega}u\Delta \bar{u}dx\\
=\mu\sum_{k=1}^{N}(u,\omega_k)\int_{\Omega}\omega_k\Delta\bar{u}dx.
\end{multline}
Integrating by parts and using the fact that $w_k$ is an eigenfunction of the $-\Delta$ under homogeneous Dirichlet boundary condition with eigenvalue $\lambda_k$, we have
\begin{multline}\label{IBP1}
\int_{\Omega}\nabla u_{t}\nabla\bar{u}dx+(\lambda +i\alpha)\int_{\Omega}|{\Delta u}|^2dx\\
+(\kappa+i\beta)\int_{\Omega}\left(\frac{p+2}{2}|u|^{p}|{\nabla u}|^2+\frac{p}{2}|u|^{p-2}u^2(\nabla \bar{u})^2\right)dx\\
-\gamma\int_{\Omega}|{\nabla u}|^2dx=-\mu\sum_{k=1}^{N}|(u,\omega_k)|^2\lambda_k,
\end{multline} where $(\nabla \bar{u})^2:=\nabla \bar{u}\cdot \nabla \bar{u}.$

Taking two times the real part of both sides of \eqref{IBP1},
\begin{multline}
\frac{d}{dt}\int_{\Omega}|{\nabla u}|^2dx+2\lambda\int_{\Omega}|{\Delta u}|^2dx\\
+2\kappa\text{Re}\int_{\Omega}\left(\frac{p+2}{2}|u|^{p}|{\nabla u}|^2+\frac{p}{2}|u|^{p-2}u^2(\nabla \bar{u})^2\right)dx\\
-2\beta\text{Im}\int_{\Omega}\left(\frac{p+2}{2}|u|^{p}|{\nabla u}|^2+\frac{p}{2}|u|^{p-2}u^2(\nabla \bar{u})^2\right)dx\\
-2\gamma\int_{\Omega}|{\nabla u}|^2dx=-2\mu\sum_{k=1}^{N}|(u,\omega_k)|^2\lambda_k.
\end{multline}

Note that since $\kappa>0$, we have

\begin{multline}\kappa\text{Re}\int_{\Omega}\left(\frac{p+2}{2}|u|^{p}|{\nabla u}|^2+\frac{p}{2}|u|^{p-2}u^2(\nabla \bar{u})^2\right)dx\\
\ge \kappa\int_{\Omega}\left(\frac{p+2}{2}-\frac{p}{2}\right)|u|^{p}|{\nabla u}|^2dx\ge 0.\end{multline}

On the other hand,
\begin{multline}\label{nonlinterm}
-\beta\text{Im}\left[\int_{\Omega}\left(\frac{p+2}{2}|u|^{p}|{\nabla u}|^2+\frac{p}{2}|u|^{p-2}u^2(\nabla \bar{u})^2\right)dx\right]\\
=-\frac{p\beta}{2}\text{Im}\int_{\Omega}|u|^{p-2}u^2(\nabla \bar{u})^2dx\le \frac{p\beta}{2}\int_{\Omega}|u|^{p}|{\nabla u}|^2dx\\
\le \frac{p\beta}{2}\|u\|_{L^{p+2}(\Omega)}^{p}\|\nabla u\|_{L^{p+2}(\Omega)}^2.
\end{multline}

We recall the following Gagliardo-Nirenberg inequalities which are true for $p\le \frac{4}{n-2}$ in dimensions $n\ge 3$ and for $p<\infty$ in dimensions $n\le 2$:

\begin{equation}\label{gag1}
  \|u\|_{L^{p+2}(\Omega)}\le C\|u\|_{H^2(\Omega)}^{\theta}\|u\|_{L^2(\Omega)}^{1-\theta} \text { where } \theta=\frac{np}{4(p+2)}\in \left(0,\frac{1}{2}\right],
\end{equation}

\begin{equation}\label{gag2}
  \|\nabla u\|_{L^{p+2}(\Omega)}\le C\|u\|_{H^2(\Omega)}^{\xi}\|u\|_{L^2(\Omega)}^{1-\xi} \text { where } \xi=\frac{(n+2)p+4}{4(p+2)}\in \left[0,\frac{1}{2}\right).
\end{equation}

Now, we use \eqref{gag1} and \eqref{gag2} in \eqref{nonlinterm}.  We get

\begin{equation}\label{nonlinterm2}
\frac{p\beta}{2}\|u\|_{L^{p+2}(\Omega)}^{p}\|\nabla u\|_{L^{p+2}(\Omega)}^2\le C\|u\|_{H^2(\Omega)}^{1+a}\|u\|_{L^2(\Omega)}^{1+b}
\end{equation} where

$$a=p\theta+2\xi-1=\frac{np}{4}\in (0,1]$$ and $$b=(1-\theta)p+2(1-\xi)-1=\frac{(4-n)p}{4}\in (-1,3].$$

It is well known that \begin{equation}\label{h2est}\|u\|_{H^2(\Omega)}\le C\|\Delta u\|_{L^2(\Omega)},\end{equation} from which it follows that

\begin{equation}
  \|u\|_{H^2(\Omega)}^{1+a}\le C\|\Delta u\|_{L^2(\Omega)}^{1+a}.
\end{equation}

Hence, the right-hand side of \eqref{nonlinterm2} is bounded by
\begin{equation}
  C\|\Delta u\|_{L^2(\Omega)}^{1+a}\|u\|_{L^2(\Omega)}^{1+b}.
\end{equation}  Combining this with $L^2-$stabilization result we can bound the above by

$$C\|\Delta u\|_{L^2(\Omega)}^{1+a}\|u_0\|_{L^2(\Omega)}^{1+b}e^{-\omega(1+b)t}.$$

Now, if $p<\frac{4}{n}$, then $1+a<2$ and the above term can be estimated as

$$C\|\Delta u\|_{L^2(\Omega)}^{1+a}\|u_0\|_{L^2(\Omega)}^{1+b}e^{-\omega(1+b)t}\le \epsilon\|\Delta u\|_{L^2(\Omega)}^2+C_\epsilon\|u_0\|_{L^2(\Omega)}^{\frac{2(1+b)}{1-a}}e^{-\frac{2\omega(1+b)}{1-a}t}$$ where $\epsilon>0$ denotes a fixed generic constant which can be chosen as small as we wish.

Hence, \eqref{nonlinterm2} becomes

\begin{multline}\label{nonlineest3}
  \frac{p\beta}{2}\|u\|_{L^{p+2}(\Omega)}^{p}\|\nabla u\|_{L^{p+2}(\Omega)}^2\\
  \le\epsilon\|\Delta u\|_{L^2(\Omega)}^2+C_\epsilon\|u_0\|_{L^2(\Omega)}^{\frac{2(1+b)}{1-a}}e^{-\frac{2\omega(1+b)}{1-a}t}
  \le \epsilon\|\Delta u\|_{L^2(\Omega)}^2+C_{u_0}e^{-\omega\zeta t},
\end{multline} where $$\zeta\equiv \frac{2(1+b)}{1-a}> 2.$$

Using \eqref{nonlineest3} and employing Parseval's identity for the derivative and using the fact that both $p$ and $\kappa$ are positive in \eqref{IBP1}, we obtain

\begin{multline}
\frac{d}{dt}\int_{\Omega}|{\nabla u}|^2dx+\left(2\lambda-\epsilon\right)\int_{\Omega}|{\Delta u}|^2dx-2\gamma\sum_{k=1}^{N}|(u,\omega_k)|^2\lambda_k\\
-2\gamma\sum_{k=N+1}^{\infty}|(u,\omega_k)|^2\lambda_k
\leq-2\mu\sum_{k=1}^{N}|(u,\omega_k)|^2\lambda_k+C_{u_0}e^{-\omega\zeta t}.
\end{multline}  Summing up the terms and assuming $\mu\ge \gamma$, we have
\begin{multline*}
\frac{d}{dt}\int_{\Omega}|{\nabla u}|^2dx+\left(2\lambda-\epsilon\right)\int_{\Omega}|{\Delta u}|^2dx\\
\leq 2(\gamma-\mu)\sum_{k=1}^{N}|(u,\omega_k)|^2\lambda_k+2\gamma\sum_{k=N+1}^{\infty}|(u,\omega_k)|^2\lambda_k+C_{u_0}e^{-\omega\zeta t}
\\\le 2\gamma\lambda_{N+1}^{-1} \|{\Delta u}\|_{L^2(\Omega)}^2+C_{u_0}e^{-\omega\zeta t}.
\end{multline*}  Hence, using the Poincaré like inequality $\|\nabla u\|_{L^2(\Omega)}^2\le \lambda_1^{-1}\|\Delta u\|_{L^2(\Omega)}^2$, we obtain

\begin{equation}
\frac{d}{dt}\int_{\Omega}|{\nabla u}(x,t)|^2dx+(\omega-\epsilon)\|{\nabla u}(t)\|_{L^2(\Omega)}^2\leq C_{u_0}e^{-\omega\zeta t}.
\end{equation}
Integrating this inequality and using the fact that $\zeta>2$, the above yields
\begin{equation}
\int_{\Omega}|{\nabla u}(x,t)|^2dx\leq e^{-(\omega-\epsilon) t}\int_{\Omega}|\nabla u_0(x)|^2dx+C_{u_0}e^{-(\omega-\epsilon) t}
\end{equation} which proves the $H^1$ decay.  Note that the decay rate can be made arbitrarily close to $\omega$ but not exactly $\omega$.

Now, let us consider the case $p=4/n$.  In this case, $1+a=2$ and we can choose $\|u_0\|_{L^2(\Omega)}$ small enough that
\begin{equation}C\|\Delta u\|_{L^2(\Omega)}^{1+a}\|u_0\|_{L^2(\Omega)}^{1+b}e^{-\omega(1+b)t}
\le \epsilon\|\Delta u\|_{L^2(\Omega)}^{2}\end{equation} for $t\ge 0$.  Now, we can complete the rest of the proof similar to the case $p<4/n$.  Hence, the same decay rate estimate also holds for this case.

\section{Steering solutions}  In this section, we consider steering solutions of the CGLE via finite parameter feedback controllers to other solutions of the CGLE.  Of course, here the feedback controller depends on the target system.  In Section \ref{secsteer1}, we prove that an appropriate finite-parameter control can steer the solutions to any desired solution of the uncontrolled system. This has been recently shown in the context of damped wave equations \cite{KTPreprint}.  In Section \ref{secsteer2}, we choose the target system slightly different in such a way that its solution exponentially decays.  We show that the controlled solution also decays exponentially.
\subsection{Steering solutions to any solution}\label{secsteer1}
Suppose that $v$ be a desired solution of the non-controlled Ginzburg-Landau model below.
\begin{equation}\label{a1noncont}
v_t-(\lambda+i\alpha)\triangle v+(\kappa+i\beta)|v|^p v-\gamma v=0,\quad x\in \Omega, t>0,
\end{equation}
\begin{equation}\label{a2noncont}
v(x,t)=0,\quad x\in \partial\Omega, t>0,
\end{equation}where $\Omega\subset \mathbb{R}^n$ is a bounded domain with regular boundary,
$\lambda,\kappa,\gamma>0,$ $\alpha,\beta\in \mathbb{R}$, and $p>-1$.\\
Our aim is to find appropriate conditions on $\lambda$ and $N$ so that the solution of the controlled problem
\begin{equation}\label{a1a}
u_t-(\lambda+i\alpha)\triangle u+(\kappa+i\beta)|u|^p u-\gamma u=-\mu\sum_{k=1}^{N}(u-v,w_k)w_k,\quad x\in \Omega, t>0,
\end{equation}
\begin{equation}\label{a2a}
u(x,t)=0,\quad x\in \partial\Omega, t>0,
\end{equation}
approaches $v$ in the long-time.

\begin{thm}[Steering I]\label{steering}Let $v$ be a solution of the non-controlled system \eqref{a1noncont}-\eqref{a2noncont} and suppose $\mu\ge \gamma$, $\displaystyle\lambda_{N+1}>\frac{\gamma}{\lambda}$, and $\kappa\ge C_p^{-1}|\beta|$ for $C_p\equiv \frac{|p|}{2\sqrt{p+1}}$.  Then the solution of the controlled system $u$ \eqref{a1a}-\eqref{a2a} must converge to $v$ as $t$ increases in the sense
$$\|u(t)-v(t)\|_{L^2(\Omega)}^2\leq e^{-\omega t}\| u_0-v_0\|_{L^2(\Omega)}^2,$$ where $\omega=2(\lambda-\gamma\lambda_{N+1}^{-1})\lambda_1>0$.\end{thm}

\begin{proof}
Subtracting \eqref{a1noncont} from \eqref{a1a}  we get

\begin{equation}\label{z1}
z_t-(\lambda+i\alpha)\triangle z+(\kappa+i\beta)(|u|^p u-|v|^pv)-\gamma z=
-\mu\sum_{k=1}^{N}(z,w_k)w_k,\quad x\in \Omega, t>0,
\end{equation}

\begin{equation}\label{z2}
z(x,t)=0,\quad x\in \partial\Omega, t>0,
\end{equation}
where $z=u-v.$


Multiplying equation \eqref{z1} by $\overline z$, integrating over $\Omega$, and then taking two times the real part:
\begin{equation}\label{en1}
\frac {d}{d t}\|z(t)\|_{L^2(\Omega)}^2+2\lambda\|\nabla z\|_{L^2(\Omega)}^2+J_0-2\gamma \|z\|^2=-2\mu\sum\limits_{k=1}^ N |(z,w_k)|^2.
\end{equation}
where
$$
J_0=2 Re\left( (\kappa+i\beta)\int_{\Omega} (|u|^p u-|v|^pv)\bar{z}dx\right)
$$
It is clear that
\begin{equation}\label{zz1}
J_0=2 \kappa Re\left( \int_{\Omega} (|u|^p u-|v|^pv)\bar{z}dx\right)-2\beta Im\left(  \int_{\Omega} (|u|^p u-|v|^pv)\bar{z}dx\right)
\end{equation}

Applying the inequality (see \cite[Lemma 2.1]{OY2002})
$${|Im\left(|u|^{p}u-|v|^{p}v,u-v\right)|}\leq\overbrace{\frac{|p|}{2\sqrt{p+1}}}^{C_p}{Re\left(|u|^{p}u-|v|^{p}v,u-v\right)},$$
we see that if $\kappa\ge |\beta|C_p^{-1}$ then $J_0\geq 0.$ Hence we obtain (using Parseval's once again)
\begin{equation}
\frac{d}{d t}\|z(t)\|_{L^2(\Omega)}^2+2\lambda\|\nabla z\|_{L^2(\Omega)}^2+2(\mu-\gamma)\sum\limits_{k=1}^ N |(z,w_k)|^2-2\gamma\sum\limits_{k=N+1}^ {\infty} |(z,w_k)|^2\leq 0
\end{equation}
Using the following inequality,
$$\sum\limits_{k=N+1}^ {\infty} |(z,w_k)|^2\leq\lambda_{N+1}^{-1}\|\nabla z\|_{L^2(\Omega)}^2$$
we get
\begin{equation}
\frac{d}{d t}\|z(t)\|_{L^2(\Omega)}^2+2\lambda\|\nabla z(t)\|_{L^2(\Omega)}^2-2\gamma\lambda_{N+1}^{-1}\|\nabla z(t)\|_{L^2(\Omega)}^2 \leq0.
\end{equation} Therefore
\begin{equation}
\frac{d}{d t}\|z(t)\|_{L^2(\Omega)}^2+2(\lambda-\gamma\lambda_{N+1}^{-1})\|\nabla z(t)\|_{L^2(\Omega)}^2 \leq0.
\end{equation}
Now, by Poincaré  inequality,
\begin{equation}\label{zineq}
\frac{d}{d t}\|z(t)\|_{L^2(\Omega)}^2+\omega\| z(t)\|_{L^2(\Omega)}^2 \leq0.
\end{equation}
Thus, we have
\begin{equation}
\|z(t)\|_{L^2(\Omega)}^2\leq e^{-\omega t}\| z_0\|_{L^2(\Omega)}^2.
\end{equation}
\end{proof}
\subsection{Steering solutions to an exponential decay}\label{secsteer2}  Note that there is no guarantee that the solution of the uncontrolled system in \eqref{a1noncont} is a decaying or even a stable solution since we do not know the relationship among the given parameters.  Therefore, a better approach might be to start with a solution which is already known to be stable, e.g., an exponentially decaying solution.  Therefore, we can start with first considering the system below
\begin{equation}\label{stable1}
v_t-(\lambda+i\alpha)\triangle v+(\kappa+i\beta)|v|^p v-\tilde{\gamma} v=0,\quad x\in \Omega, t>0,
\end{equation} where $\tilde{\gamma}<\lambda\lambda_1$, again under the homogeneous Dirichlet boundary condition.  It is easy to see that multiplying \eqref{stable1} by $\bar{v}$, integrating over $\Omega$, and taking two times the real part, one obtains
\begin{equation*}
  \frac{d}{dt}\int_\Omega |v|^2 dx+ 2\lambda \int_\Omega|\nabla v| ^2 dx+2\kappa\int_\Omega|v|^{p+2} dx-2\tilde{\gamma}\int_\Omega|v|^2dx = 0.
\end{equation*}  Now, using the Poincaré  ineqality and the fact that $\kappa>0$, it follows that
\begin{equation*}
  \frac{d}{dt}\int_\Omega |v|^2 dx + 2\left(\lambda \lambda_1-\tilde{\gamma}\right) \int_\Omega|v| ^2 dx\le 0,
\end{equation*}  from which it is easy to deduce the exponential decay estimate $$\|v(t)\|_{L^2(\Omega)}\le \|v_0\|_{L^2(\Omega)}e^{-\left(\lambda \lambda_1-\tilde{\gamma}\right)t}$$ for $t\ge 0$.  Now, let us consider the feedback control system \eqref{a1a} where $v$ is a solution of \eqref{stable1} instead of \eqref{a1noncont}.  In this case, \eqref{zineq} takes the form
$$\frac{d}{d t}\|z(t)\|_{L^2(\Omega)}^2+\omega\| z(t)\|_{L^2(\Omega)}^2 \leq -2(\tilde{\gamma}-\gamma)\text{Re}\int_\Omega v\bar{z}dx.$$  Now, using $\epsilon$-Young's inequality at the right-hand side, we get
\begin{equation*}
  \frac{d}{d t}\|z(t)\|_{L^2(\Omega)}^2+\omega\| z(t)\|_{L^2(\Omega)}^2 \leq \epsilon \|z(t)\|_{L^2(\Omega)}^2 + \frac{|\tilde{\gamma}-\gamma|^2}{\epsilon} \|v(t)\|_{L^2(\Omega)}^2.
\end{equation*} where $\epsilon>0$ is a fixed, small number.  Multiplying both sides by $e^{(\omega-\epsilon)t}$, using the decay estimate on $v$, and then integrating over $(0,t)$, we obtain
\begin{equation*}
  \|z(t)\|_{L^2(\Omega)}^2 \le e^{-(\omega-\epsilon)t}\left(\|z_0\|_{L^2(\Omega)}^2+\frac{|\tilde{\gamma}
  -\gamma|^2}{\epsilon}\|v_0\|_{L^2(\Omega)}^2\overbrace{\int_0^te^{(\tilde{\omega}-\epsilon)s}ds}^{\frac{1}{(\tilde{\omega}-\epsilon)}[e^{(\tilde{\omega}-\epsilon)t}-1]}\right),
\end{equation*}  where $\tilde{\omega}:=2(\tilde{\gamma}-\gamma\lambda_{N+1}^{-1}\lambda_1)$.  That is,
\begin{multline}
  \|z(t)\|_{L^2(\Omega)}^2 \\
  \le \|z_0\|_{L^2(\Omega)}^2e^{-(\omega-\epsilon)t}+\frac{|\tilde{\gamma}
  -\gamma|^2}{\epsilon}\|v_0\|_{L^2(\Omega)}^2{\frac{1}{(\tilde{\omega}-\epsilon)}[e^{2(\tilde{\gamma}-\lambda\lambda_1)t}-e^{-(\omega-\epsilon)t}]}
\end{multline} for $t\ge 0$.

Now, there are two cases: Case I ($\tilde{\omega}>0$) and Case II ($\tilde{\omega}\le 0$).  Case I is satisfied if $\tilde{\gamma}>\gamma\lambda_{N+1}^{-1}\lambda_1$ and Case II is satisfied if $\tilde{\gamma}\le\gamma\lambda_{N+1}^{-1}\lambda_1$.  It is easy to observe that the previous assumption $\displaystyle\lambda_{N+1}>\frac{\gamma}{\lambda}$ implies $\gamma\lambda_{N+1}^{-1}\lambda_1<\lambda\lambda_1$. Now, recalling that we also have $\tilde{\gamma}<\lambda\lambda_1$, we can say that
Case I is satisfied if $\tilde{\gamma}\in \left(\gamma\lambda_{N+1}^{-1}\lambda_1,\lambda\lambda_1\right)$ and Case II is satisfied if $\tilde{\gamma}\in \left(-\infty,\gamma\lambda_{N+1}^{-1}\lambda_1\right]$.
Hence, we have the following theorem.
\begin{thm}[Steering II]\label{steering2} Let $v$ be an (exponentially decaying) solution of \eqref{stable1} where $\tilde{\gamma}<\lambda\lambda_1$ and suppose $\mu\ge \gamma$, $\displaystyle\lambda_{N+1}>\frac{\gamma}{\lambda}$, and $\kappa\ge C_p^{-1}|\beta|$ for $C_p\equiv \frac{|p|}{2\sqrt{p+1}}$.  Then the solution $u$ of the controlled system \eqref{a1a}-\eqref{a2a} must converge to $v$ as $t$ increases.  More precisely,

\begin{enumerate}
  \item If $\tilde{\gamma}\in \left(\gamma\lambda_{N+1}^{-1}\lambda_1,\lambda\lambda_1\right)$, then \begin{multline}\|u(t)-v(t)\|_{L^2(\Omega)}^2\\
      \leq e^{-(\omega-\epsilon)t}\|u_0-v_0\|_{L^2(\Omega)}^2+\frac{|\tilde{\gamma}
  -\gamma|^2}{\epsilon}\|v_0\|_{L^2(\Omega)}^2{\frac{1}{(\tilde{\omega}-\epsilon)}e^{2(\tilde{\gamma}-\lambda\lambda_1)t}},\end{multline}
  \item If $\tilde{\gamma}\in \left(-\infty,\gamma\lambda_{N+1}^{-1}\lambda_1\right]$, then   \begin{multline}\|u(t)-v(t)\|_{L^2(\Omega)}^2\\
      \leq e^{-(\omega-\epsilon)t}\|u_0-v_0\|_{L^2(\Omega)}^2+\frac{|\tilde{\gamma}
  -\gamma|^2}{\epsilon}\|v_0\|_{L^2(\Omega)}^2{\frac{1}{(\epsilon-\tilde{\omega})}e^{-(\omega-\epsilon)t}}\end{multline}
\end{enumerate}

 for $t\ge 0$, where $\epsilon$ is a fixed (can be chosen arbitrarily small) number, $\omega=2(\lambda-\gamma\lambda_{N+1}^{-1})\lambda_1$, where $\tilde{\omega}=2(\tilde{\gamma}-\gamma\lambda_{N+1}^{-1}\lambda_1)$.
\end{thm}
\begin{cor}\label{cordecay}In particular, the controlled solution in Theorem \ref{steering2} exponentially decays to zero under the same assumptions.  More precisely, \begin{enumerate}
  \item If $\tilde{\gamma}\in \left(\gamma\lambda_{N+1}^{-1}\lambda_1,\lambda\lambda_1\right)$, then \begin{multline}\|u(t)\|_{L^2(\Omega)}^2\\
      \leq e^{-(\omega-\epsilon)t}\|u_0-v_0\|_{L^2(\Omega)}^2+\|v_0\|_{L^2(\Omega)}^2\left(\frac{|\tilde{\gamma}
  -\gamma|^2}{\epsilon}\frac{1}{(\tilde{\omega}-\epsilon)}+1\right){e^{2(\tilde{\gamma}-\lambda\lambda_1)t}},
  \end{multline}
  \item If $\tilde{\gamma}\in \left(-\infty,\gamma\lambda_{N+1}^{-1}\lambda_1\right]$, then   \begin{multline}\|u(t)\|_{L^2(\Omega)}^2\\
      \leq e^{-(\omega-\epsilon)t}\left[\|u_0-v_0\|_{L^2(\Omega)}^2+\frac{|\tilde{\gamma}
  -\gamma|^2}{\epsilon}\|v_0\|_{L^2(\Omega)}^2{\frac{1}{(\epsilon-\tilde{\omega})}}\right]\\
  +\|v_0\|_{L^2(\Omega)}^2e^{-2\left(\lambda\lambda_1-\tilde{\gamma}\right)t}\end{multline}
\end{enumerate}

for $t\ge 0$, where $\epsilon$ is a fixed (can be chosen arbitrarily small) number, $\omega=2(\lambda-\gamma\lambda_{N+1}^{-1})\lambda_1$, where $\tilde{\omega}=2(\tilde{\gamma}-\gamma\lambda_{N+1}^{-1}\lambda_1)$.\end{cor}
  \begin{proof}Follows by Theorem \ref{steering2} and the inverse triangle inequality $$\left|\|u(t)\|_{L^2(\Omega)}-\|v(t)\|_{L^2(\Omega)}\right|\le \|u(t)-v(t)\|_{L^2(\Omega)}.$$\end{proof}

\section{$L^2$-stabilization with nodal observables}
In this seciton, we consider the Ginzburg-Landau equation where the right-hand side is considered as a feedback control described by finitely many nodal valued observables.

\begin{equation}\label{a1n}
u_t-(\lambda+i\alpha)\triangle u+(\kappa+i\beta)|u|^p u-\gamma u=-\mu\sum_{k=1}^{N}hu(\bar{x}_k)\delta(x-x_k),\quad x\in \Omega= (0,L), t>0,
\end{equation}
\begin{equation}\label{a2n}
u(0,t)=u(L,t)=0,\quad, t>0,
\end{equation} with $x_k, \bar{x}_k\in J_k$.

We prove the following theorem.
\begin{thm}  Let $u$ be a solution of  \eqref{a1n}-\eqref{a2n} with $\lambda\ge \mu h^2$ and $\displaystyle \frac{\mu}{4}>\gamma$.  Then,
$$\|u(t)\|_{L^2(\Omega)}\le e^{-\left[\lambda_1\left(\lambda-{\mu h^2}\right)+\left(\frac{\mu}{4}-\gamma\right)\right]t}\|u_0\|_{L^2(\Omega)}$$ for $t\ge 0.$
\end{thm}
\begin{proof}

We first compute the action of the $H^{-1}$ functionals at both sides of \eqref{a1n} on $u\in H_0^1(0,L)$, so we have
\begin{multline}\label{L2iden}
\frac{1}{2}\frac{d}{dt}\|u(t)\|_{L^2(\Omega)}^2+\lambda\|u_x(t)\|_{L^2(\Omega)}^2+\kappa\|u(t)\|_{L^{p+2}(\Omega)}^{p+2}-\gamma\|u(t)\|_{L^2(\Omega)}^2\\
=-\mu h\text{Re}\sum_{k=1}^{N}u(\bar{x}_k)\bar{u}(x_k).
\end{multline}
Writing $u(\bar{x}_k)\bar{u}(x_k)=\left(u(\bar{x}_k)-u(x_k)\right)\bar{u}(x_k)+|u(x_k)|^2$, using the basic inequality $|a\cdot b|\le \frac{|a|^2}{2}+\frac{|b|^2}{2}$, and employing \cite[Lemma 6.1]{AT2014} we obtain
\begin{multline}
\left|-\mu h\text{Re}\sum_{k=1}^{N}u(\bar{x}_k)\bar{u}(x_k)\right|\le \frac{\mu h}{2}\sum_{k=1}^{N}|u(\bar{x}_k)-u(x_k)|^2-\frac{\mu h}{2}\sum_{k=1}^{N}|u(x_k)|^2\\
\le {\mu h^2}\|u_x(t)\|_{L^2(\Omega)}^2-\frac{\mu}{4}\|u(t)\|_{L^2(\Omega)}^2.
\end{multline}
Using the above estimate in \eqref{L2iden} yields
\begin{multline}
\frac{1}{2}\frac{d}{dt}\|u(t)\|_{L^2(\Omega)}^2+\lambda\|u_x(t)\|_{L^2(\Omega)}^2+\kappa\|u(t)\|_{p+2}^{p+2}-\gamma\|u(t)\|_{L^2(\Omega)}^2\\
\le {\mu h^2}\|u_x(t)\|_{L^2(\Omega)}^2-\frac{\mu}{4}\|u(t)\|_{L^2(\Omega)}^2.
\end{multline}
Summing up the terms and dropping the term $\kappa\|u\|_{p+2}^{p+2}$, we obtain
\begin{equation}
  \frac{1}{2}\frac{d}{dt}\|u(t)\|_{L^2(\Omega)}^2+\left(\lambda-{\mu h^2}\right)\|u_x(t)\|_{L^2(\Omega)}^2+\left(\frac{\mu}{4}-\gamma\right)\|u(t)\|_{L^2(\Omega)}^2\le 0.
\end{equation}

Assuming $\lambda\ge \mu h^2$, $\displaystyle \frac{\mu}{4}>\gamma$, and using the Poincaré inequality, we obtain
\begin{equation}
  \frac{1}{2}\frac{d}{dt}\|u(t)\|_{L^2(\Omega)}^2+\left[\lambda_1\left(\lambda-{\mu h^2}\right)+\left(\frac{\mu}{4}-\gamma\right)\right]\|u(t)\|_{L^2(\Omega)}^2\le 0,
\end{equation} which implies
$$\|u(t)\|_{L^2(\Omega)}\le e^{-\left[\lambda_1\left(\lambda-{\mu h^2}\right)+\left(\frac{\mu}{4}-\gamma\right)\right]t}\|u_0\|_{L^2(\Omega)}$$ for $t\ge 0.$
\end{proof}

\bibliographystyle{amsplain}
\bibliography{thereferences}
\end{document}